\documentclass[12pt,twoside, reqno]{amsart}
\usepackage[all]{xy}   
\usepackage {amssymb,latexsym,amsthm,amsmath}
\usepackage[colorlinks, linkcolor=blue, citecolor=red]{hyperref}
\usepackage[utf8]{inputenc}
\usepackage{enumitem} 
\usepackage{tikz}
\usepackage{float}
\usetikzlibrary{arrows.meta}
\usepackage{xcolor}
\usepackage{todonotes}
\usepackage{graphicx}

\long\def\symbolfootnote[#1]#2{\begingroup%
\def\thefootnote{\fnsymbol{footnote}}\footnote[#1]{#2}\endgroup}

\newcommand{\Z}{\ensuremath{\mathbb{Z}}}

\newcommand{\R}{\mathbb R}

\def \Homeo {\mathrm{Homeo}}

\def \S {\mathbb{S}}

\def \Id {\mathrm{Id}}
\def \E {\mathcal{E}}

\newtheorem{theorem}{Theorem}[section]
\newtheorem{lemma}[theorem]{Lemma}
\newtheorem{corollary}[theorem]{Corollary}
\newtheorem{proposition}[theorem]{Proposition}
\newtheorem*{theorem*}{Theorem}
\theoremstyle{definition}
\newtheorem{remark}[theorem]{Remark}
\newtheorem{definition}[theorem]{Definition}

\newtheorem{question}[theorem]{Question}

\numberwithin{equation}{section}

\newcommand{\ignore}[1]{}

\newcommand{\mynote}[1]{}
\newcommand{\secref}[1]{Section~\ref{#1}}
\newcommand{\thmref}[1]{Theorem~\ref{#1}}
\newcommand{\lemref}[1]{Lemma~\ref{#1}}
\newcommand{\remref}[1]{Remark~\ref{#1}}
\newcommand{\propref}[1]{Proposition~\ref{#1}}
\newcommand{\corref}[1]{Corollary~\ref{#1}}

\newcommand{\figref}[1]{Figure~\ref{#1}}
\newcommand{\defref}[1]{Definition~\ref{#1}}
\begin{document}
\setcounter{section}{0}
\setcounter{tocdepth}{1}
\title[Conjugacy of free mappings]{Conjugacy of free mappings embedded in a flow}
\author[Sushil Bhunia]{Sushil Bhunia}
\author[Gangotryi Sorcar]{Gangotryi Sorcar}
\address{Department of Mathematics, BITS-Pilani, Hyderabad Campus, Hyderabad, India}
\email{sushilbhunia@gmail.com}
\address{University of Delaware, Newark, DE 19716, USA}
\email{gangotryi@gmail.com}
\thanks{}
\today
\begin{abstract}
In this paper we study free mappings of the plane, that is orientation preserving fixed point free homeomorphisms of $\R^2$. We provide a necessary and sufficient condition under which two free mappings of the plane that are embedded in flows are conjugate to one another using Haefliger-Reeb theory of plane foliations.
\end{abstract}
\maketitle

\section*{Introduction}
In this article we study orientation preserving fixed point free homeomorphisms of the plane that are embedded in flows and relate the algebraic property of conjugacy between two such homeomorphisms with the topological property of equivalence of the plane foliations induced by these homeomorphisms. To do this, we invoke Haefliger-Reeb theory which associates a possibly non-Hausdorff one-dimensional manifold to a plane foliation. Since $\mathbb{R}^2$ is the universal cover of all oriented surfaces other than the sphere, we think it is a fundamental object to study in order to carry out similar investigations on homeomorphism groups of surfaces in future. Conjugacy and reversibility of elements of any group $G$ are defined as follows: Two elements $f$ and $g$ in $G$ are said to be {\it conjugate} in $G$ if there is some $h$ in $G$ such that $g=hfh^{-1}$. An element $g$ in $G$ is said to be \emph{reversible} in $G$ if $g$ and $g^{-1}$ are conjugate in $G$. An element $g\in G$ is called \textit{strongly reversible} if $g^{-1}=xgx^{-1}$ for some $x\in G$ with $x^2=1$. Clearly, a strongly reversible element is reversible. Note that reversibility (resp. strongly reversibility) is a conjugacy invariant.

Orientation preserving homeomorphisms without any fixed points are called \emph{free mappings}. In \cite[Proposition 8.5]{fs15}, O'Farrell and Short prove that two free mappings $f$ and $g$ of $\mathbb{R}$ are conjugate if either $f(x)<x$ and $g(x)<x$ or $f(x)>x$ and $g(x)>x$ for all $x\in \mathbb{R}$. This nice characterization of the conjugacy of free mappings of the real line is due to the fact that the order topology and the usual topology on the real line are equivalent. Consequently, a free mapping $h$ of the real line either pushes all the points to the right ($h(x)>x$ for any real number $x$) or it pushes all the points to the left ($h(x)<x$ for any real number $x$). This property has been used crucially in the methods of O'Farrell and Short in \cite[Chapter 8]{fs15} and \cite{fs55}. In fact, they were able to obtain similar results for the group of 
homeomorphisms on $\mathbb{S}^1$ since $\mathbb{S}^1$ inherits its topology from $\mathbb{R}$ (for details, see \cite{goas09} and \cite[Chapter 9]{fs15}).
However, in the context of free mappings of the plane the methods stated above no longer apply since we do not have an ordering on $\mathbb{R}^2$ that is compatible with the usual topology of $\mathbb{R}^2$. So we look at the situation from a different perspective. The information of whether $h(x)>x$ or $h(x)<x$ for any real number $x$ can also be encoded by assigning an orientation on $\mathbb{R}$, which in turn can be regarded as an oriented codimension-$0$ foliation of $\mathbb{R}$ (think of this as an arrow on the real line indicating which way $h$ is pushing the points on it). This seemingly complicated way of looking at a simple matter, generalizes to an efficient set up in $\mathbb{R}^2$. With each free mapping $h$ of $\mathbb{R}^2$ that is embedded in a flow (Definition \ref{embeddable}), we can associate an oriented codimension-$1$ foliation $F(h)$ of $\mathbb{R}^2$. In \secref{prel}, we explain this in detail along with the notion of equivalence of two oriented codimension-$1$ plane foliations (\defref{foliationequiv}). From here on, we will simply say foliation instead of codimension-$1$ foliation. Each free mapping of $\R^2$ also partitions $\R^2$ into equivalence classes called {\it fundamental regions} (for details see \secref{freemap}). We can now state our main theorem, which establishes a topological characterization of when two free mappings of $\mathbb{R}^2$ are conjugate.  

\begin{theorem*}[\thmref{mainthm1}]
	Let $f$ and $g$ be two free mappings of $\mathbb{R}^2$, with finitely many fundamental regions, that are embedded in flows. Then $f$ is conjugate to $g$ or $f$ is conjugate to $g^{-1}$ in $\Homeo^+(\mathbb{R}^2)$ if and only if their corresponding oriented plane foliations are equivalent.
\end{theorem*} 
\noindent
\emph{Idea of the proof:} Since $f$ and $g$ are embedded in flows, they induce two oriented plane foliations $F(f)$ and $F(g)$ respectively (see \propref{freetofoliation}). If $F(f)$ and $F(g)$ are equivalent, then the corresponding leaf spaces $V_f$ and $V_g$ are homeomorphic by Haefliger-Reeb theory (see \lemref{equivfoliations}). Therefore $V_f\times\mathbb{S}^1$ and $V_g\times\mathbb{S}^1$ are homeomorphic as well. Given any free mapping $f$ embedded in a flow, we show that $\mathbb{R}^2/\langle f\rangle$ is a  fiber bundle over the leaf space $V_f$ with $\mathbb{S}^1$ fibers (see \lemref{fiberbundle}). Furthermore, we show that $V_f$ is contractible (see \lemref{contractable}) making this a trivial fiber bundle, i.e., $\mathbb{R}^2/\langle f\rangle$ is homeomorphic to $V_f\times \mathbb{S}^1$. Now since $\mathbb{R}^2/\langle f\rangle$ is homeomorphic to $\mathbb{R}^2/\langle g\rangle$, using covering space argument we show that $f$ is conjugate to $g$ or $g^{-1}$ (see \propref{homeoconj}). For the converse, we show that if $g=hfh^{-1}$ for some $h$ in $\Homeo^{+}(\R^2)$, then $h$ transports the leaves of $F(f)$ to the leaves of $F(g)$.

\section{Preliminaries}\label{prel}
In this section we fix some notations and terminology which will be used throughout this article. Let $\Homeo(\mathbb{R}^2)$ denote the group (under composition) of all homeomorphisms from $\R^2$ to itself with the identity homeomorphism $\Id$. It has an index two normal subgroup (denoted by $\Homeo^+(\mathbb{R}^2)$) consisting only of the orientation preserving homeomorphisms. For any $f$ in $\Homeo(\R^2)$, we denote the fixed point set of $f$ as $\mathrm{Fix}(f):=\{x\in \R^2 : f(x)=x\}$. An element $f$ in $\Homeo^+(\R^2)$ is called a \emph{free mapping} if $\mathrm{Fix}(f)=\emptyset$ (i.e., it is fixed point free). Free mappings are also known as Brouwer homeomorphisms. Throughout, we will use $\langle f \rangle$ to denote the group generated by $f$ and $\simeq$ to replace the phrase ``is homeomorphic to".

\subsection{Free mappings}\label{freemap}
To explore some properties of free mappings, we will follow a nice treatise on the topic by Andrea \cite{andrea67}. For easiness of reading, we will reiterate some of the basic definitions and results from \cite{andrea67}.
The very first observation, due to Brouwer, is that any point $x$ on the plane diverges on repeated applications of a free mapping $f$ and its inverse. That is, $f^n(x) \rightarrow \infty$ as $n\to \pm\infty$ for any $x\in \mathbb{R}^2$. Andrea gives an equivalence relation on the plane by relating points $x$ and $y$ that ``diverge together". They call this equivalence relation {\emph {codivergence}}. To explain this properly, we need to define the divergence of a subset $A$ of $\mathbb{R}^2$. We say  a subset $A$ of $\mathbb{R}^2$ diverges and write $f^n(A)\rightarrow \infty$ as $n \to \pm\infty$, if and only if for all compact sets $K\;(\subset\R^2)$ one has $f^n(A)\cap K=\emptyset$ for all but finitely many $n$. Now, let $f$ be a free mapping and $x, y\in \R^2$. We say that $x$ and $y$ are \emph{codivergent} (with respect to $f$), denoted by $x\sim_{f} y$, if there exists 
 a curve $\gamma$ joining $x$ and $y$ such that   $f^n(\gamma)\rightarrow \infty$ as $ n\rightarrow \pm \infty$. 
Note that $\sim_{f}$ is an equivalence relation on $\R^2$.  
The equivalence classes, denoted by $\{x\}_f$, are called the \textit{fundamental regions} of $f$. For topological properties of the fundamental regions the reader may refer to \cite{jones72}. Note that if $f$ is a free mapping, then for any $h$ in $\Homeo(\R^2)$, $hfh^{-1}$ is also a free mapping because $\mathrm{Fix}(hfh^{-1})=h(\mathrm{Fix}(f))$.  

\begin{definition}\label{embeddable}\cite{andrea67}
    A free mapping $f$ is said to be \emph{embedded in the flow} $\{\bar{f}^\rho\}$ if:
    \begin{enumerate}
        \item For every real number $\rho$, $\bar{f}^\rho$ is a homeomorphism of $\R^2$.
        \item The function $\bar{f}^\rho(p)$ is jointly continuous in the variables $\rho$ and $p$.
        \item $\bar{f}^\alpha(\bar{f}^\beta(p))=\bar{f}^{\alpha+\beta}(p)$; and
        \item the homeomorphism $\bar{f}^1$ in the flow is the same as our free mapping $f$.
    \end{enumerate}
\end{definition}

\begin{remark}\label{properflowline}
    Due to conditions (4) and (3) in the above definition setting $\alpha=1$ $\beta=0$ we can deduce that $\bar{f}^0$ is the identity homeomorphism. Then we can further deduce that $\bar{f}^n=f^n$ for all integers $n$. From \cite[Proposition 2.1]{andrea67} we obtain a few important facts about flows that will be useful for us. First we see that the homeomorphism $\bar{f}^\rho$ has no fixed points in $\R^2$ if $\rho\ne0$. Next, for any $x$ in $\R^2$ if we denote $L_x^{\bar{f}}= \{\bar{f}^\rho(x): \rho \in \R\}$, we can see that  $L_x^{\bar{f}}$ is invariant under $f$, $L_x^{\bar{f}}$ is a homeomorphic image of $\mathbb{R}$ in $\mathbb{R}^2$ passing through $x$, and $L_x^{\bar{f}}\cup \{\infty\}$ is a Jordan curve on the sphere $\R^2\cup\{\infty\}$. Moreover, for any $x, y$ in $\R^2$, $L_x^{\bar{f}}$ and $L_y^{\bar{f}}$ are either disjoint or equal. Now since each $L_x^{\bar{f}}$ is homeomorphic to $\R$, there is an ordering on $L_x^{\bar{f}}$ and for any two points $x_1$ and $x_2$ on $L_x^{\bar{f}}$ we will denote $x_1<x_2$ if $x_1$ comes before $x_2$ according to that ordering. One can check that $\bar{f}^{\rho}(x)<\bar{f}^{\lambda}(x)$ whenever $\rho<\lambda$. Below when we write {\emph {$\bar{f}$ is order preserving on $L_x^{\bar{f}}$}}, we are referring to this property.
\end{remark}

\begin{lemma}\label{darlinglemma}
Let $f$ be a free mapping embedded in two flows $\{\bar{g}^\rho\}$ and $\{\bar{f}^\rho\}$. Then we have $L_x^{\bar{g}}=L_x^{\bar{f}}$ for all $x\in \R^2$.
\end{lemma}

\begin{proof}
    Denote the portion of $L_x^{\bar{f}}$ between $x$ and $f(x)$ including both $x$ and $f(x)$ as $[x,f(x)]$. It is sufficient to prove that $[x,f(x)]\subset L_x^{\bar{g}}$. By way of contradiction, assume that there is some $z$ in $[x,f(x)]$ such that $z$ is not in $L_x^{\bar{g}}$, that is $L_x^{\bar{g}}\ne L_z^{\bar{g}}$. Consider the sets $A=\{y \in [x,f(x)]: L_y^{\bar{g}}\cap L_z^{\bar{g}}\ne \emptyset\}$ and $B=\{y \in [x,f(x)]: L_y^{\bar{g}}\cap L_x^{\bar{g}}\ne \emptyset\}$.  For any $y$ in $[x,f(x)]$ the curve $L_y^{\bar{g}}$ must intersect $L_x^{\bar{g}}$ or $L_z^{\bar{g}}$ since $\bar{f}$ is order preserving on $L_x^{\bar{f}}$ and $\bar{f}(y)=\bar{g}(y)$. Therefore, we have $A\cup B=[x,f(x)]$. For any $y$ in $[x,f(x)]$ the curve $L_y^{\bar{g}}$ cannot intersect both $L_x^{\bar{g}}$ and $L_z^{\bar{g}}$ since $L_x^{\bar{g}}\ne L_z^{\bar{g}}$. Therefore, we also have $A\cap B=\emptyset$. Now by definition $A=L_z^{\bar{g}}\cap [x,f(x)]$ and $B=L_x^{\bar{g}}\cap [x,f(x)]$, which implies $A$ and $B$ are closed sets. Since $[x,f(x)]$ is connected, we have arrived at a contradiction. 
\end{proof}

\begin{definition}
    For any free mapping $f$ embedded in a flow $\{\bar{f}^\rho\}$ we define $L_x^f=\{\bar{f}^\rho(x): \rho \in \R\}$ to be the \emph{flowline of $f$ passing through $x$}. This definition is valid because by the lemma above, $L_x^f$ is independent of the flow in which $f$ is embedded. 
\end{definition}
\subsection{Foliations of $\mathbb{R}^2$}
\begin{definition}\cite[Definition 1]{lawson74}
    A \emph{foliation} $F$ of $\mathbb{R}^2$ is a decomposition of $\mathbb{R}^2$ into a union of disjoint connected subsets $\{L_a\}_{a\in A}$, called the leaves of the foliation, with the following property: for every point $p$ in $\mathbb{R}^2$ there is a neighborhood $U$ containing $p$ and a homeomorphism $x=(x_1,x_2):U\rightarrow \mathbb{R}^2$ such that for each leaf $L_a$, the coordinate function $x_2$ restricted to $U\cap L_a$ is a constant depending on $a$.
\end{definition}

 Note that any foliation of $\R^2$ can be oriented, refer to \cite[pages 15-16]{s22} for details. One can think of this as putting compatible arrows on each leaf of the foliation. Below are some pictures of oriented foliations of $\R^2$.
 
\begin{figure}[H]
\begin{center}
    \includegraphics[width=150mm,scale=3.5]{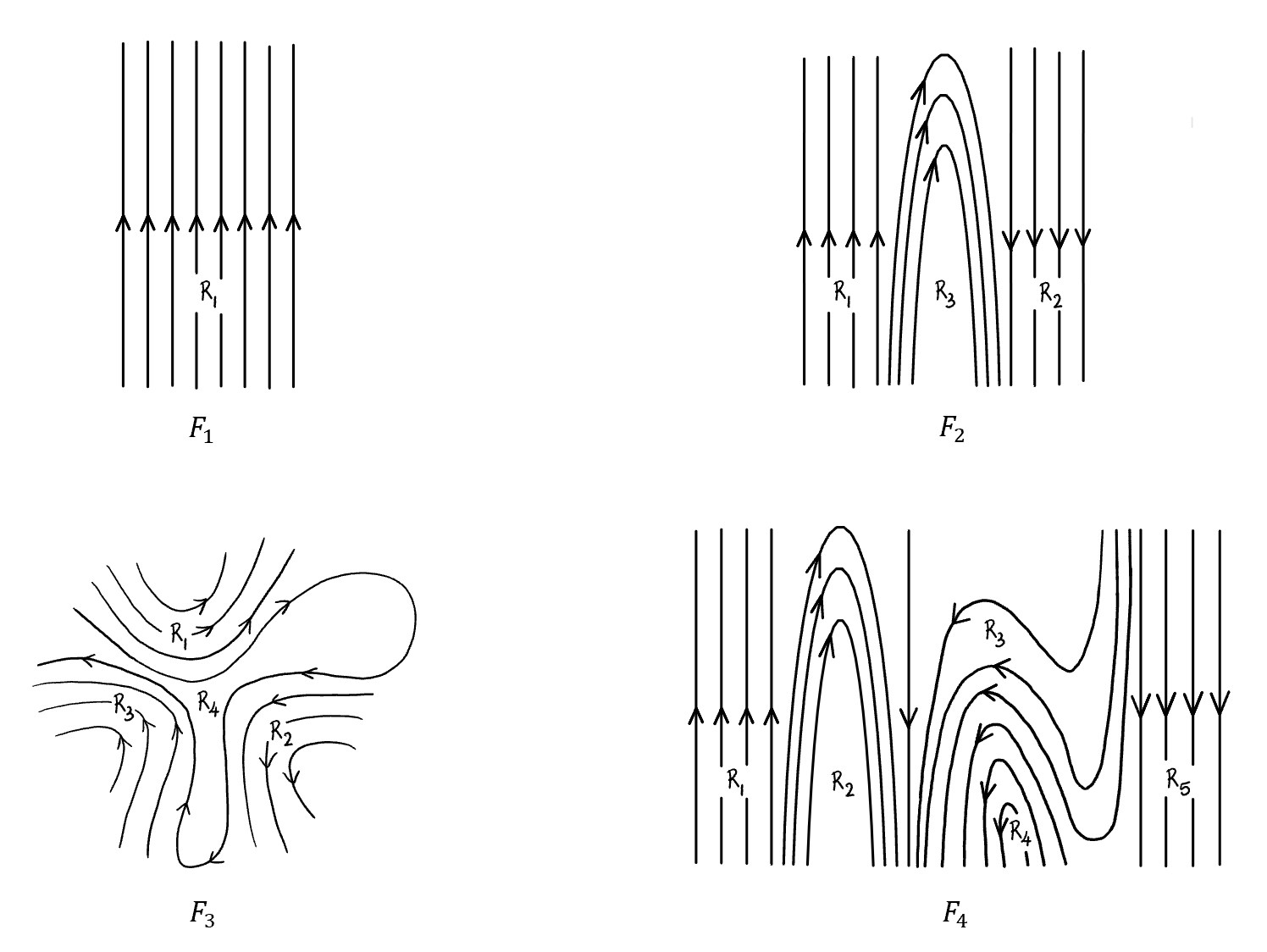}
    \caption{Oriented foliations}
    \end{center}
\end{figure}

\begin{definition}\label{foliationequiv}
    Two oriented foliations $F$ and $F'$ are said to be \emph{equivalent} if there is a homeomorphism of $\R^2$ that transports the oriented leaves of $F$ to those of $F'$.
\end{definition}

For any foliation $F$ of $\R^2$, define a relation $\approx_F$ on $\mathbb{R}^2$ such that for any $x, y \in \R^2$, $x \approx_F y$ if and only if $x$ and $y$ are in the same leaf of $F$. This is an equivalence relation on $\R^2$ and the equivalence class of any $x$ in $\R^2$ is denoted by $[x]_{\approx_F}$.

\begin{definition}\label{leafspace}
    Given any foliation $F$, the quotient space $\R^2/\approx_F$ is called the \emph{leaf space} of $F$ and is denoted by $V_F$.
\end{definition}
We are stating the following proposition without proof from \cite{hr57} or \cite[Proposition 3]{s22}:

\begin{proposition}\label{foltoleaf}
    Let $F$ be a foliation of $\R^2$. The leaf space $V_F$ is a $1$-dimensional simply connected manifold, not necessarily Hausdorff, with countable basis. 
\end{proposition}
\begin{definition}\label{branchpts}
 A point $x$ in a manifold $M$ is said to be a \emph{branch point} if there is a point $y\in M$ ($x\neq y$) such that any neighborhood of $x$ has a non-empty intersection with any neighborhood of $y$. In this case we say $x$ and $y$ are non-separable.
\end{definition}
Note that the leaf space $V_F$ can be oriented. Then one can define an order relation among non-separable branch points of $V_F$ (see \cite[Page 16]{s22}) using the orientation of the foliation. Below we provide pictures of the leaf spaces of the oriented foliations shown above with this ordering:

\begin{figure}[H]
\centering
    \includegraphics[width=150mm,scale=4.2]{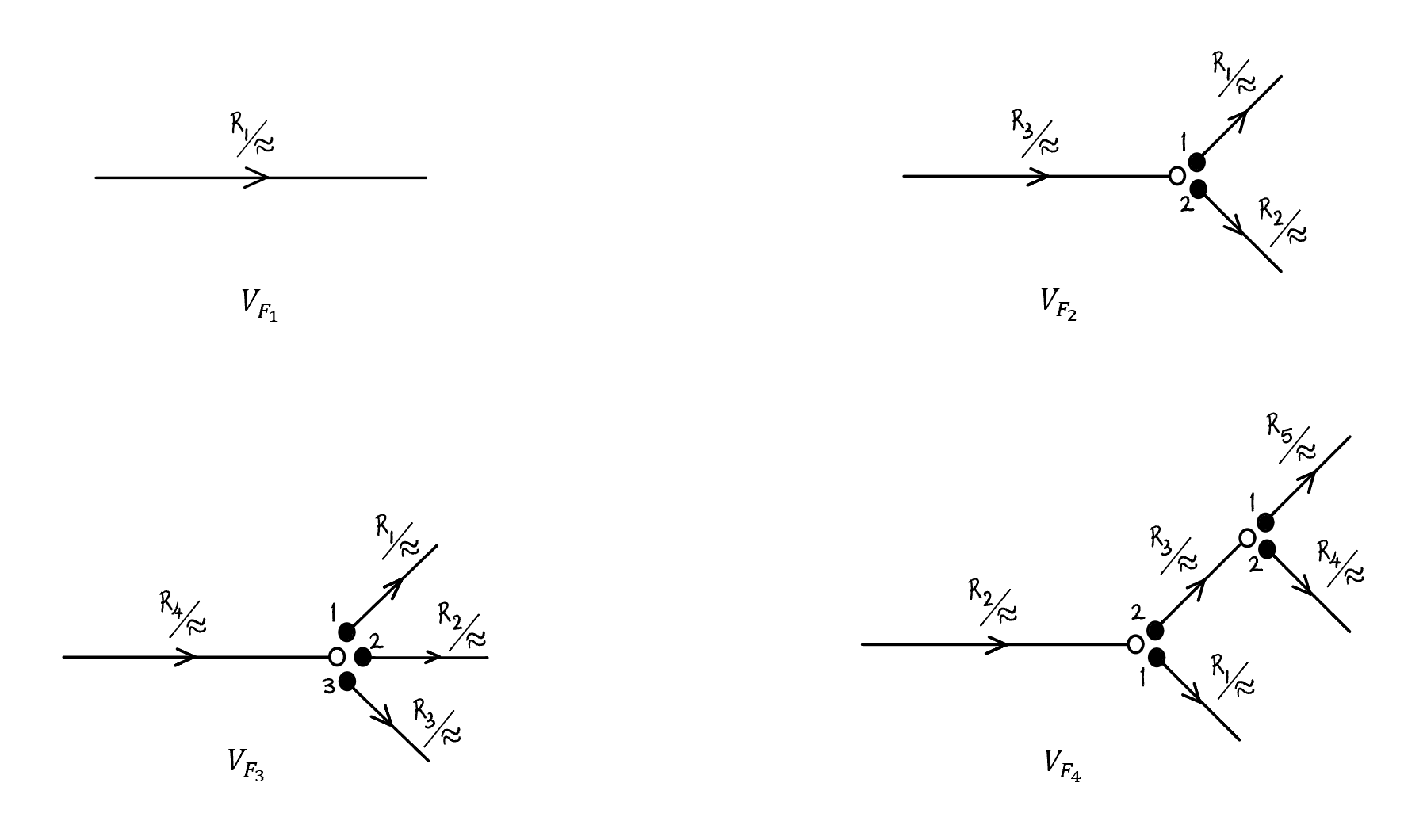}
    \caption{Leaf spaces}
    \label{figure2}
\end{figure}

Below is an important result that characterizes equivalent foliations in terms of their leaf spaces. We will use crucially in the proof of our main theorem. 
\begin{lemma}\cite[Page 16]{s22}\label{equivfoliations}
Two oriented foliations $F$ and $F'$ of $\R^2$ are equivalent if and only if there exists a homeomorphism from $V_F$ to $V_{F'}$ that preserves the ordering.
\end{lemma}

The proposition below explains how a free mapping embedded in a flow induces a foliation on the plane.

\begin{proposition}\label{freetofoliation}
    Given a free mapping $f$ embedded in a flow, its flowlines  $\{L^f_x : x\in\R^2\}$ form a foliation of $\R^2$. This foliation will be denoted as $F(f)$. 
\end{proposition}
The corresponding leaf space of $F(f)$ will be simply denoted by $V_f$ instead of $V_{F(f)}$. Before we prove this proposition we need to recall some facts. 

\begin{definition}\cite[Definition 2.1]{jf92}\label{domtrans}
    A \emph{domain of translation} of a free mapping $f$ is an open connected subset of $\mathbb{R}^2$ whose boundary is $L \cup f(L)$ where $L$ is the image of a proper embedding of $\mathbb{R}$ in $\R^2$ such that $L$ separates $f(L)$ and $f^{-1}(L)$.
\end{definition}

Note that in this definition $L$ is not a flowline of $f$, in fact a flowline will not satisfy the requirements of $L$ in this definition.
Now, we mention a theorem (without proof) and a remark which we will need to prove \propref{freetofoliation}.

\begin{theorem}\cite[Brouwer's Plane Translation Theorem 2.2]{jf92}\label{brouwertrans}
    Let $f$ be a free mapping of $\R^2$. Every point $x$ in $\R^2$ is contained in some domain of translation. 
\end{theorem}

\begin{remark}\label{conjtrans}
    If $D$ is the closure of a domain of translation of $f$, then $U=\bigcup f^n(D)$, for all integers $n$, is an open set invariant under $f$ and the restriction of $f$ to $U$ is conjugate to a translation of $\R^2$ \cite[Page 220]{jf92}. 
\end{remark}
\subsection*{Proof of \propref{freetofoliation}} 
    We will show that $\{L^f_x : x\in\R^2\}$ satisfies the definition of foliation above. First, we note that $L^f_x$ is a connected set for every $x$ in $\mathbb{R}^2$ by definition of a flowline. Next, $L^f_x$ is disjoint from $L^f_y$ unless $L^f_x=L^f_y$ since every $x$ in $\mathbb{R}^2$ is contained in a unique flowline.
    
    By \thmref{brouwertrans}, for any $x \in \R^2$ there is a domain of translation $W$ of $f$ containing $x$. Let $U=\bigcup f^n(D)$ for all integers $n$, where $D$ is the closure of $W$. By \remref{conjtrans}, $f(U)=U$ and there is a homeomorphism $h: U \rightarrow \R^2$ such that $hfh^{-1}:\R^2\rightarrow\R^2$ is a translation. Since all translations are conjugate to each other in $\mathrm{Homeo}(\R^2)$ (see the proof of \lemref{trans}), we can assume that $hfh^{-1}$ is the translation that shifts every point on the plane $1$ unit to the right. We will denote this translation by $T_{1,0}$. Now we argue that for the point $x$, the homeomorphism $h=(h_1,h_2):U\rightarrow \R^2$ fits the requirements of the definition of foliation. First note that for any flowline $L^f_y$, the function $hfh^{-1}(h(L^f_y))=h(L^f_y)$. Since $hfh^{-1}=T_{1,0}$ is a translation, and the only $1$-dimensional submanifolds of the plane that a translation can leave invariant are its flowlines, we conclude that $h(L^f_y)$ must be a horizontal line. Hence, $h_2$ is a constant function depending on the flowline $L^f_y$. \qed
    
We end this section with the following definition which will be used later.

\begin{definition}\label{csa}
	A group $G$ acting on a topological space $X$ is said to be a \emph{covering space action} if every $x \in X$ has a neighborhood $U_x \subset X$ such that $g(U_x) \cap U_x=\emptyset$ for all $g\ne e$, where $e$ is the identity element of $G$.
\end{definition}

\section{Some basic results on free mappings}
We begin with the following basic fact on free mappings:
\begin{lemma}\label{region}
	Let $f$ be a free mapping and $h$ in $\Homeo(\R^2)$. Then the following hold for all $x$ in $\R^2$. 
	\begin{enumerate}[leftmargin=*]
		\item \label{equi1} $\{x\}_f=\{x\}_{f^{-1}}$.
		\item\label{equi2} $x\sim_{f} y$ if and only if $h(x)\sim_{hfh^{-1}} h(y)$ $($ i.e., $h(\{x\}_f)=\{h(x)\}_{hfh^{-1}}$ $)$.
	\end{enumerate}
\end{lemma}

\begin{proof}
	\begin{enumerate}[leftmargin=*]
		\item By definition, $\{x\}_f=\{y\in \R^2 : x\sim_{f} y\}$. Suppose that $\gamma$ is a curve joining $x$ and $y$. Then $f^n(\gamma)\rightarrow \infty$  as $n\rightarrow \pm\infty$ if and only if $(f^{-1})^{-n}(\gamma)\rightarrow \infty$ as $-n\rightarrow \mp\infty.$ Hence, the proof.
		\item Let $\gamma$ be a curve joining $x$ and $y$ as above, then $h\circ\gamma=h(\gamma)$ is a curve joining $h(x)$ and $h(y)$. Note that  $(hfh^{-1})^n(h(\gamma))=(hf^nh^{-1})(h(\gamma))=hf^n(\gamma)$. Now we will show that if $f^n(\gamma)\rightarrow \infty$ as $ n\rightarrow \pm\infty$, then $hf^n(\gamma)\rightarrow \infty$ as $n\rightarrow \pm\infty$. By way of contradiction, suppose that $hf^n(\gamma)\nrightarrow \infty$ as $n\rightarrow \infty$ or $n\rightarrow -\infty$. Without loss of generality, suppose $hf^n(\gamma)\nrightarrow \infty$ as $n\rightarrow \infty$. Then there exists a compact subset $K$ of $\R^2$ such that $K\cap hf^n(\gamma)\neq \emptyset$ for all $n\in\mathbb{N}$. Therefore for all $n\in\mathbb{N}$,
		\begin{align*}
		h^{-1}(K\cap hf^n(\gamma))&\neq \emptyset\\
		\Rightarrow h^{-1}(K)\cap f^n(\gamma)&\neq \emptyset.
		\end{align*}
		This contradicts the fact that $f^n(\gamma)\rightarrow \infty$ (as $n\rightarrow \pm\infty$) since $h^{-1}(K)$ is compact. Hence the claim. To show the converse one can do a similar computation with $h(\gamma)$ replaced by $h^{-1}(\gamma')$ where $\gamma'$ is a curve joining $h(x)$ and $h(y)$. 
	\end{enumerate}
\end{proof}
As an immediate corollary, we have:
\begin{corollary}\label{samefun}
	Let $f$ and $g$ be free mappings having $m$ and $n$ fundamental regions respectively. If $f$ and $g$ are conjugate in $\Homeo(\R^2)$, then $m=n$. 
\end{corollary}

\begin{lemma}\label{permutation}
	Suppose that $f$ is a free mapping which is reversible in $\Homeo^+(\R^2)$, then a reverser of $f$ permutes the fundamental regions of $f$.
\end{lemma}

\begin{proof}
	Since $f$ is reversible in $\Homeo^+(\R^2)$ there exists a reverser $h$ in $\Homeo^+(\R^2)$ such that $f^{-1}=hfh^{-1}.$
	Now for $x$ in $\R^2$ by \lemref{region}\eqref{equi2}) we have  
	$\{x\}_{f^{-1}}=\{x\}_{hfh^{-1}}=\{h(h^{-1}(x))\}_{hfh^{-1}}
	=h(\{h^{-1}(x)\}_f)$.
	Therefore by putting $x=h(y)$ we get $\{h(y)\}_{f^{-1}}=h(\{y\}_f)$. So in view of  \lemref{region}\eqref{equi1}, $\{h(y)\}_f=h(\{y\}_f)$ for all $y\in \R^2$. 
\end{proof}

\begin{lemma}\label{trans}
	Translations are strongly reversible elements in $\mathrm{Homeo}^+(\mathbb{R}^2)$, i.e., they can be reversed by an involution.
\end{lemma}

\begin{proof}
	For two fixed real numbers 
	$a$ and $b$, let $T_{a,b}$ denote the translation that maps $(x,y)$ 
	to $(x+a,y+b)$ for all $(x,y)$ in  $\mathbb{R}^2$. First, we will show that 
	any two non-trivial translations $T_{a,b}$ and $T_{c,d}$ are conjugate. 
	Conjugating both translations by a single suitable rotation we 
	may assume that $a,b,c,$ and $d$ are all non-zero reals. Then
	$h^{-1}T_{a,b}h=T_{c,d}$ where $\displaystyle h(x,y)=\left(ax/c,by/d \right)$. 
	As a special case of this, we observe that the antipodal map, also an orientation preserving involution of $\R^2$, given by $f_A: (x,y)\mapsto (-x,-y)$ conjugates 
	any translation to its inverse, i.e., $f_AT_{a,b}f_A^{-1}=T_{a,b}^{-1}$.
	Therefore translations are strongly reversible elements 
	of $\Homeo^+(\R^2)$. 
\end{proof}

 Now we explore how the fundamental regions of a free mapping affect its conjugacy class and thereby its reversibility. By \corref{samefun}, conjugate free mappings with finite fundamental regions have the same number of fundamental regions. The converse of this statement is not true in general. However, it is true for free mappings with only one fundamental region. Andrea \cite[Theorem 4.1]{andrea67} shows that a free mapping is conjugate to a translation if and only if it has exactly one fundamental region. Therefore by \lemref{trans} any free mapping with one fundamental region is reversible. A free mapping with exactly two fundamental regions does not exist (see \cite[Proposition 3.2]{andrea67}). For free mappings with three or more fundamental regions, same number of fundamental regions is not enough to guarantee conjugacy. To see an example of this phenomenon we will study the Reeb flow $f$ (see \cite[Page 67, Example 1]{andrea67}) which has exactly three fundamental regions denoted as follows: $R_{-1}=\{(x,y)\in \R^2: x\le -1\}$, $R_{0}=\{(x,y)\in \R^2: -1< x < 1\}$, and $R_{1}=\{(x,y)\in \R^2: x\ge 1\}$. The restriction of $f$ on $R_{-1}$ is $T_{(0,1)}$ and the restriction of $f$ on $R_{-1}$ is $T_{(0,-1)}$. The exact formula for the restriction of $f$ on $R_0$ will not be needed for our purposes. From the diagram below we can see that $f$ is reversible in $\Homeo(\R^2)$ by the reflection about the $y$-axis, denoted by $r$. However, the following lemma shows that $f$ is not reversible in $\Homeo^+(\R^2)$.

\begin{figure}[H]
\centering
    \includegraphics[width=90mm,scale=.5]{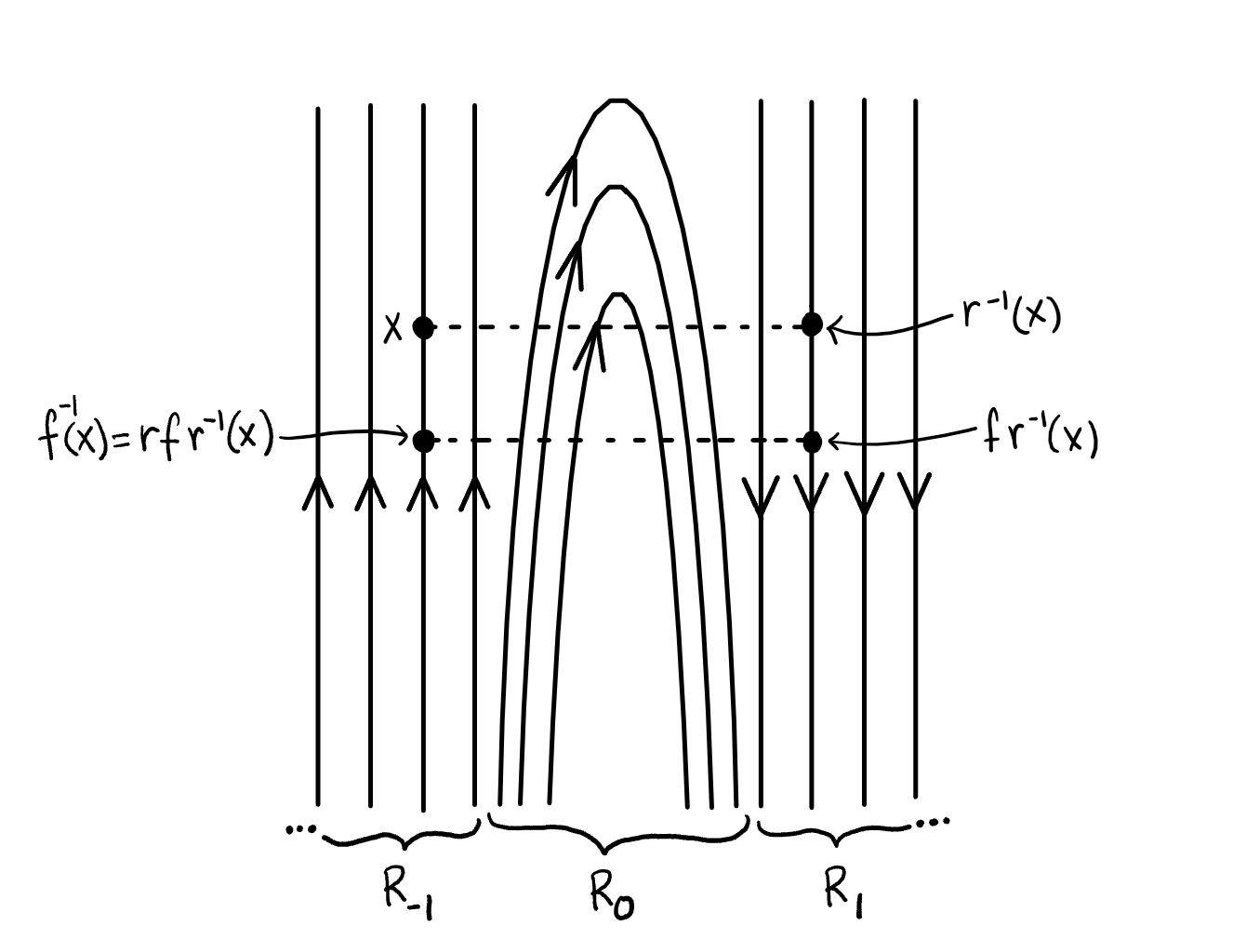}
    \caption{Reversibility of the Reeb flow in $\Homeo(\R^2)$}
\end{figure}

\begin{lemma}\label{reebflow}
	The Reeb flow is not reversible in $\Homeo^+(\R^2)$.
\end{lemma}
\begin{proof}
    Let $f$ be the Reeb flow as described above with the three fundamental regions $R_{-1}, R_{0}$, and $R_{1}$.
	If possible suppose that $f$ is reversible, i.e., $f^{-1}=hfh^{-1}$ for some $h\in\Homeo^+(\R^2)$. 
	By \lemref{permutation}, $h$ permutes $R_{-1}, R_0, R_1$.
	Note that $R_0$ is an open subset of $\R^2$ while $R_{-1}$ and $R_1$ are closed subsets of $\R^2$. Being a homeomorphism $h$ leaves $R_0$ invariant as it is the only open fundamental region and therefore $h$ either swaps $R_{-1}$ and $R_1$ or leaves them invariant. The following two cases take care of this dichotomy.
	
	\noindent
	\textbf{Case 1:} Let $h(R_{-1})=R_{-1}$ and $h(R_1)=R_1$. Let $(x,y)\in R_{-1}$ and let $h(x,y)=(\bar{x}, \bar{y})$. 
	\begin{align*}
	f^{-1}(\bar{x},\bar{y})&=hfh^{-1}(\bar{x},\bar{y})\\
	(\bar{x},\bar{y}-1)&=hf(x,y)\\
	(\bar{x},\bar{y})-(0,1)&=h(x,y+1)\\
	 h(x,y)-(0,1)&=h(x,y+1)\\
	 h(x,y)-h(x,y+1)&=(0,1).
	\end{align*}
	Therefore $h(x,y)-h(f(x,y))=(0,1)$. Now let $p$ be any point in $R_0$. The diagram below explains the contradiction since $h$ is an orientation preserving homeomorphism.
	
	\begin{figure}[H]
    \centering
    \includegraphics[width=90mm,scale=.5]{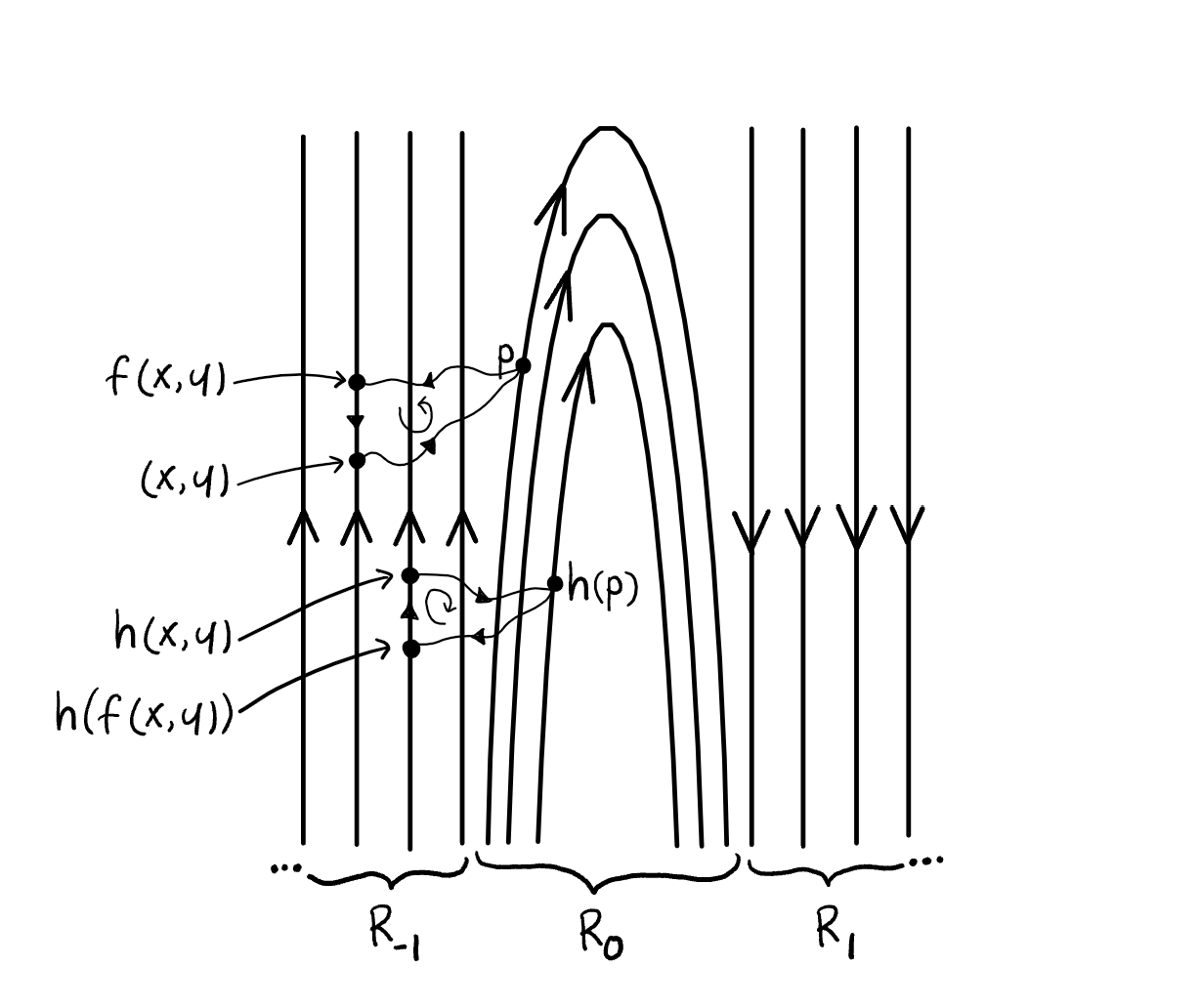}
    \caption{Non reversibility of the Reeb flow in $\Homeo^{+}(\R^2)$}
    \end{figure}
	
	\noindent		
	\textbf{Case 2:} Let $h(R_{-1})=R_1$ and $h(R_1)=R_{-1}$. Let $(x,y)\in R_{-1}$ and let $h(x,y)=(\bar{x}, \bar{y})$ in $R_1$. 
	\begin{align*}
	f^{-1}(\bar{x},\bar{y})&=hfh^{-1}(\bar{x},\bar{y})\\
	(\bar{x},\bar{y}+1)&=hf(x,y)\\
	(\bar{x},\bar{y})+(0,1)&=h(x,y+1)\\
	h(x,y)+(0,1)&=h(x,y+1)\\
	h(x,y+1)-h(x,y)&=(0,1).
	\end{align*}
	Thus we have $h(f(x,y))-h(x,y)=(0,1)$ for all $(x,y)\in R_{-1}$. Again we arrive at a contradiction as in Case 1.
\end{proof}
\section{Proof of the main results}\label{mainsec}
 We have seen in the previous section that having the same number of fundamental regions is not sufficient for conjugacy, we have obtained the following result that indicates what is. 

\begin{theorem}\label{mainthm1}
 Suppose $f$ and $g$ are two free mappings of $\R^2$, with finitely many fundamental regions, that are embedded in flows. Then either $f$ is conjugate to $g$ or $f$ is conjugate to $g^{-1}$ in $\Homeo^{+}(\R^2)$ if and only if their corresponding oriented plane foliations are equivalent.
 \end{theorem}
The following results will be used in the proof of this theorem. We begin with the following lemma.
\begin{lemma}\label{fiberbundle}
If $f$ is a free mapping embedded in a flow, then $\mathbb{R}^2/\langle f\rangle$ is a fiber bundle over the leaf space $V_f$ with $\mathbb{S}^1$ as fibers.
\end{lemma}
\begin{proof}
Note that $\mathbb{R}^2/\langle f\rangle=\{[x]_f : x\in \mathbb{R}^2\}$, where $[x]_f=\{f^n(x) : n\in \mathbb{Z}\}$. Recall that the leaf space $V_f=\{[x]_{\approx_{F(f)}} : x\in \mathbb{R}^2\}$, corresponding to the foliation $F(f)$, is a $1$-dimensional simply-connected possibly non-Hausdorff manifold. Here, $[x]_{\approx_{F(f)}}=\{y : L_x^f=L_y^f\}$. Also, recall that $L_z^f$ denotes the unique flowline of $f$ passing through $z$. Let $P$ and $Q$ be the quotient maps defined by $P(x)=[x]_f$ and $Q(x)=[x]_{\approx_{F(f)}}$. Now we can define a map $\varphi_f:\mathbb{R}^2/\langle f\rangle\rightarrow V_f$ by $\varphi_f([x]_f)=[x]_{\approx_{F(f)}}$. Therefore, we have $Q=\varphi_f\circ P$ resulting in the following commutative diagram.
\[\xymatrix{
	\R^2 \ar[dr]^{Q} \ar[d]_{P} &  \\
	\R^2/\langle f\rangle \ar[r]_{\varphi_f} &V_f.
}\] 
  Thus, $\varphi_f^{-1}([x]_{\approx_{F(f)}})=Q^{-1}([x]_{\approx_{F(f)}})/\langle f\rangle$ for all $[x]_{\approx_f}\in V_f$. Note that $Q^{-1}([x]_{\approx_{F(f)}})\simeq \mathbb{R}$ which in turn implies that $\varphi_f^{-1}([x]_{\approx_{F(f)}})\simeq \mathbb{R}/\mathbb{Z}\simeq \mathbb{S}^1$.
  
  Now, we will show that $\varphi_f: \R^2/\langle f \rangle \rightarrow V_f$ is a fiber bundle, i.e., for every $[x]_{\approx_{F(f)}}$ in $V_f$ there is an open neighborhood $U_x$ in $V_f$ containing $[x]_{\approx_{F(f)}}$ and a homeomorphism $h_x: \varphi_f^{-1}(U_x)\rightarrow U_x \times \S^1$ or  $h_x: P(Q^{-1}(U_x))\rightarrow U_x \times \S^1$. Let $[x]_{\approx_{F(f)}}$ be a point in $V_f$. Consider some $w$ in $L_x^f$. By definition of foliation, there is a coordinate neighborhood $V_w$ in $\R^2$ containing $w$ and a coordinate function which is a homeomorphism, $v=(v_1,v_2): V_w \rightarrow \R^2$ such that for each flowline $L_z^f$ the function $v_2$ restricted to $L_z^f \cap V_w$ is a constant depending on $[z]_{\approx_{F(f)}}$. This means all flowlines passing through $V_w$ are mapped to horizontal lines in $\R^2$ by the map $v$. Let $D_{v(w)}$ be an open disk in $\R^2$ containing $v(w)$. Consider the open set $v^{-1}(D_{v(w)})$ in $\R^2$ and then denote the projection $Q(v^{-1}(D_{v(w)}))$ by $U_x$. We will show that this $U_x$ will give us the local trivialization condition we need (see \figref{figure5}). One can find an injective continuous map $\alpha:(0,1)\rightarrow D_{v(w)}$ such that $\alpha(0,1)$ lies on a diameter of $D_{v(w)}$ and on some straight line $x=k$, where $k$ is a constant. 
  For any $z$ in $Q^{-1}(U_x)$, the flowline $L_z^f$ intersects $v^{-1}(D_{v(w)})$. Moreover, $L_z^f$ intersects $v^{-1}(\alpha (0,1))$ exactly once  because the horizontal line $v(L_z^f \cap V_w)$ intersects $\alpha(0,1)$ exactly once. Let us denote this point of intersection by $z_0$. Let $f$ be embedded in the flow $\{\bar{f}^\rho\}$, then since $z$ and $z_0$ are on the same flowline there exists a real number $t(z)$ such that $\bar{f}^{t(z)}(z_0)=z$. Now we define $h_x: P(Q^{-1}(U_x))\rightarrow U_x \times \S^1$ as $h_x(P(z))=(Q(z),e^{i2\pi t(z)})$ which is the required homeomorphism.

  	\begin{figure}[H]
    \centering
    \includegraphics[width=140mm,scale=4.5]{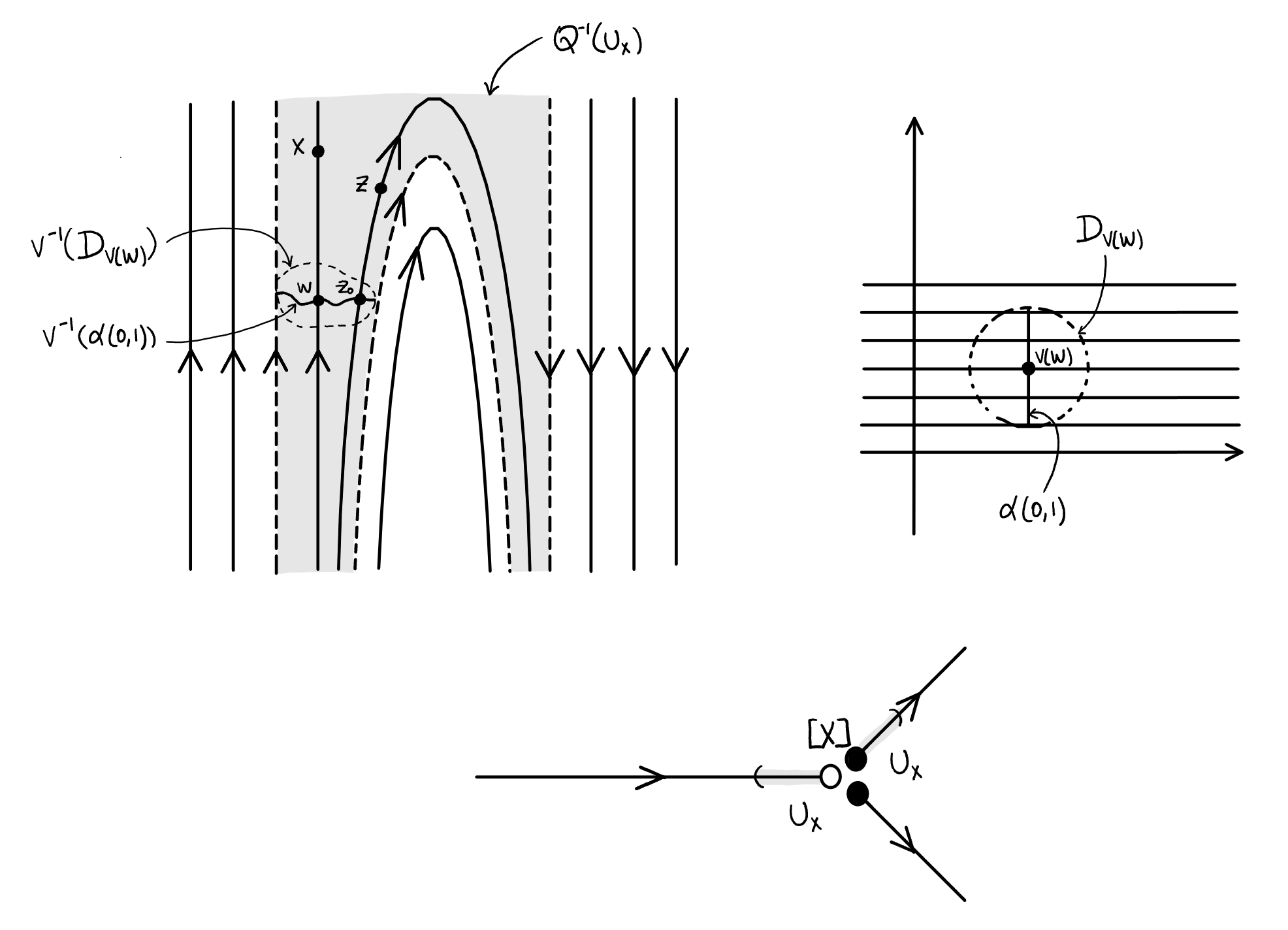}
    \caption{Illustration of the local trivialization}
    \label{figure5}
    \end{figure}
	\end{proof}
 We will now show that the above fiber bundle is trivial by showing that its base space $V_f$ is contractible. First we prove a lemma which helps us analyze the structure of $V_f$ better. Next, we set up some definitions which are essential for proving the contractibility of $V_f$.

\begin{lemma}\label{twotypes}
Let $f$ be a free mapping which is embedded in a flow. Then either $[x]_f\simeq \R$ or $\mathrm{int}([x]_f)\simeq \R^2$ for all $x\in \R^2$.
\end{lemma}

\begin{proof}
 If $\mathrm{int}([x]_f)$ is empty, then $[x]_f$ is a flowline of $f$ (see \cite[proof of Remark 9]{jones72}) and hence homeomorphic to $\R$. Now suppose $\mathrm{int}([x]_f)$ is nonempty. By the Riemann mapping theorem, any nonempty simply connected open subset of $\R^2$ is homeomorphic to $\R^2$ and by \cite[Theorem 3.2]{polish}, $\mathrm{int}([x]_f)$ is simply connected.
\end{proof}
For an arbitrary free mapping $f$ embedded in a flow, its leaf space $V_f$ is almost like a tree with modified notions of edges and vertices. Let $\{R_\lambda:\lambda \in \Lambda\}$ be the set of fundamental regions of $f$ whose interiors are homeomorphic to $\mathbb{R}^2$, where $\Lambda$ is some index set. Let $Q:\mathbb{R}^2\rightarrow V_f$ be the quotient map corresponding to the equivalence relation $\approx_{F(f)}$. Recall that $V_f$ is supplied with the quotient topology via the map $Q$ from the usual topology in $\mathbb{R}^2$. For each $\lambda$ in $\Lambda$, we will denote $Q(\mathrm{int}(R_\lambda))$ by $e_\lambda$ and call it an edge of $V_f$. We will denote by $\mathcal{E}$ the set of all edges of $V_f$, that is, $\mathcal{E}=\{e_\lambda : \lambda \in \Lambda\}$. Now let $\{R_\gamma:\gamma \in \Gamma\}$ be the set of fundamental regions that are homeomorphic to $\mathbb{R}$. We define as vertices of $V_f$ all points $Q(L)$ where $L$ is a flowline that bounds a fundamental region $R_\lambda$ for $\lambda \in \Lambda$. Note that this includes all points $Q(R_\gamma)$ for $\gamma \in \Gamma$. Let $\mathcal{V}$ denote the set of all vertices of $V_f$. It can be seen that $\mathcal{V}$ is equal to the set of all branch points (see \defref{branchpts}) of $V_f$. Note that mutually non-separable branch points in $V_f$ correspond to flowlines that bound a common fundamental region but the converse is not true. One can revisit these diagrams (see Figure \ref{figure2}) to look at the edges, vertices, and branch points of $V_f$. 

Let $e_\lambda$ be an edge in $\E$ and $x$ be any point in $e_\lambda$. Note that $e_\lambda \setminus \{x\}$ has two components. Also note that the boundary of $e_\lambda$ is equal to the boundary of $e_\lambda \setminus \{x\}$. We will say two boundary points of $e_\lambda$ are related if they are boundary points of the same component of $e_\lambda\setminus \{x\}$. This is an equivalence relation that partitions the set of boundary points of $e_\lambda$ into two disjoint sets. We will denote these two sets as $e_\lambda^-$ and $e_\lambda^+$. Now, for any edge $e_\lambda$, fix a homeomorphism $h_\lambda:(0,1)\rightarrow e_\lambda$. If the boundary of $h_\lambda(0,1/2)$ has nonempty intersection with $e_\lambda^-$ then we keep this homeomorphism otherwise compose it with an orientation reversing self-homeomorphism of $(0,1)$ and rename the composition as $h_\lambda$. 

\begin{definition}
    An edge $e_\lambda$ in $\E$ is called an {\it extreme edge} if either $e_\lambda^-$ or $e_\lambda^+$ is the empty set.
\end{definition}

\begin{definition}
    An extreme edge will be called a {\it first order extreme edge} if it has only one boundary point, otherwise it will be called a {\it second order extreme edge}. 
\end{definition}
For example in \figref{figure2}, the leaf space $V_{F_4}$ has three first order extreme edges, namely, $R_1/\approx, R_4/\approx, R_5/\approx$ and one second order extreme edge $R_2/\approx$.

\begin{lemma}\label{contractable}
If $V_f$ has finitely many edges, then it is contractible.
\end{lemma}

\begin{proof}
We will show that there is a deformation retraction of $V_f$ onto a proper subset of $V_f$ with fewer edges by making the first and second order extreme edges ``collapse" in that order. We will then prove that $V_f$ is contractible by repeating this process. 

Let $e_\lambda$ be a first order extreme edge with boundary point $x_\lambda$. Since all flowlines are Jordan curves, the flowline $L=Q^{-1}(x_\lambda)$ bounds another fundamental region $R_\mu$ distinct from $R_\lambda$. Therefore, the corresponding edge $e_\mu$ has $x_\lambda$ as one of its boundary points. Now, using the homeomorphisms $h_\mu:(0,1)\rightarrow e_\mu$ and $h_\lambda:(0,1)\rightarrow e_\lambda$ we can construct a homeomorphism $h_{\lambda\mu}: (0,2)\rightarrow e_\lambda\cup\{x_\lambda\}\cup e_\mu$. Using the deformation retraction of $(0,2)$ onto $(0,1)$, we obtain a deformation retraction of $e_\lambda\cup\{x_\lambda\} \cup e_\mu$ onto $e_\mu$. We call the last deformation retraction, collapsing the first order extreme edge $e_\lambda$.  

Let $e_\mu$ be a second order extreme edge. If $V_f$ has more than one edge, then there exists a fundamental region $R_\lambda$ that shares a boundary flowline $L$ with $R_\mu$. Therefore, the edge $e_\lambda$ has a boundary point $x_\lambda=Q(L)$ which is also a boundary point of $e_\mu$. Now, similarly as above we obtain a deformation retraction of $e_\mu\cup\{x_\lambda\}  \cup e_\lambda$ onto $e_\lambda$. We call this deformation retraction, collapsing the second order extreme edge $e_\mu$.  

Also, we are allowed to collapse a second order extreme edge only after all the first order extreme edges have been collapsed. Note that the $e_\mu$ needed to collapse a first order extreme edge $e_\lambda$ could itself be a second order extreme edge. This is why collapsing all the first order extreme edges before collapsing any second order extreme edge is necessary.

Once we have done all possible second order moves, the resulting space has fewer edges than $V_f$. Call this new space $W_1$. We can now define the concepts of extreme edge, first order and second order extreme edge in $W_1$ in a similar way as they were defined in $V_f$. We then proceed to collapse first order and second order extreme edges of $W_1$. The resulting space will be called $W_2$. We will keep proceeding like this to get a nested sequence of spaces denoted by $W_k$ for $k=1,2,3,\ldots$. Since $V_f$ has finitely many edges, there exists a natural number $n$ such that $W_n$ contains only one edge. Because otherwise we can collapse an extreme edge of $W_n$ and continue the process. Finally, one can contract the one edge in $W_n$ to a point. Therefore, our nested sequence is as follows where each $W_i$ deformation retracts onto $W_{i+1}$:
$$V_f=W_0\supset W_1 \supset W_2 \supset \cdots W_n \supset W_{n+1}=\textit{pt.}$$

Hence, $V_f$ is contractible.
\end{proof}

\begin{lemma}\label{cover}
	If $f$ is a free mapping of $\mathbb{R}^2$, then $\langle f \rangle$ is a covering space action on $\mathbb{R}^2$.
\end{lemma}

\begin{proof}
	For any $x \in \R^2$, let $y$ denote $f(x)$. We know that $y \ne x$ because $f$ is a free mapping. Now we choose two neighborhoods $U_x'$ and $U_y$ of $x$ and $y$ respectively such that $U_x'\cap U_y=\emptyset$. Because $f$ is continuous, there is a neighborhood $U_x''$ of $x$ such that $f(U_x'') \subset U_y$. Denote the intersection $U_x' \cap U_x''$ by $U_x$. We get $f(U_x)\cap U_x=\emptyset$ because $f(U_x)\subset U_y$. By \cite[Proposition 1.1]{andrea67}, if $f(K)\cap K=\emptyset$, then $f^n(K)\cap K=\emptyset$ for all $n\in \mathbb{Z}$ and all compact connected subset $K$ of $\mathbb{R}^2$. We can assume $U_x$ is connected and then by replacing $K$ with the closure of $U_x$ we obtain $f^n(U_x)\cap U_x=\emptyset$ for all $n\in \mathbb{Z}$. Thus by \defref{csa}, $\langle f\rangle$ is a covering space action on $\R^2$.
\end{proof}

\begin{proposition}\label{homeoconj}
Let $f$ and $g$ be two free mappings such that $\R^2/\langle f\rangle$ and $\R^2/\langle g\rangle$ are homeomorphic. Then $f$ is conjugate to either $g$ or $g^{-1}$ in $\Homeo^{+}(\mathbb{R}^2)$.
\end{proposition}
\begin{proof}
Let $P: \R^2\rightarrow \R^2/\langle f\rangle$ and $Q: \R^2\rightarrow \R^2/\langle g\rangle$ be two coverings with $\langle f\rangle$ and $\langle g\rangle$ being their corresponding covering groups (see \propref{cover}). Suppose that $h:\R^2/\langle f\rangle\rightarrow\R^2/\langle g\rangle$ is the given  homeomorphism. Then $h\circ P$ is a covering map with $\langle g\rangle$ being the covering group. Since $h\circ P$ and $Q$ are two covering maps for $\R^2/\langle g\rangle$ with the universal cover $\R^2$, then there exists a homeomorphism $U:\R^2\rightarrow \R^2$ such that the following diagram commutes:
\[\xymatrix{
	\R^2 \ar[r]^{U} \ar[d]_{P} & \R^2\ar[d]^{Q} \\
	\R^2/\langle f\rangle \ar[r]_{h} &\R^2/\langle g\rangle.
}\] 
That is we get,  
\begin{align}\label{eqn22}
Q\circ U=h\circ P.
\end{align}

\noindent
\textbf{Claim:} $Uf^nU^{-1}\in \langle g\rangle$ for all $n\in \Z$.

Note that $P\circ f=P$, then for all $n\in \Z$, we have 
\begin{align}\label{eqn21}
P\circ f^n=P.
\end{align}
\noindent
\textit{Proof of claim:} From Equations \eqref{eqn22} and \eqref{eqn21} we get the following: 
\begin{align*}
Q\circ(Uf^nU^{-1})&=(Q\circ U)\circ f^n\circ U^{-1}=(h\circ P)\circ f^n\circ U^{-1}\\
&=h\circ (P\circ f^n)\circ U^{-1}=h\circ P\circ U^{-1}=Q\circ U\circ U^{-1}=Q.
\end{align*}
Thus, $Uf^nU^{-1}$ is a covering transformation of $Q$ and hence the claim. 

Now consider the map $\Psi: \langle f\rangle\rightarrow \langle g\rangle$ given by $\Psi(f^n)=Uf^nU^{-1}$. Note that $\Psi$ is a group isomorphism. Under this isomorphism $\Psi(f)=UfU^{-1}$. So either $UfU^{-1}=g$ or $UfU^{-1}=g^{-1}$. 

Now, we will explain why we may assume that $U$ is orientation preserving. First note that $\R^2/\langle f \rangle$ and $\R^2/\langle g \rangle$ have canonical orientations coming from a fixed orientation on $\R^2$ and the two covering maps $P$ and $Q$. Moreover, these choice of orientations make the covering maps $P$ and $Q$ orientation preserving. Now, with respect to these orientations, $h$ is either orientation preserving or reversing. If $h$ is orientation reversing then we can compose it with an orientation reversing self-homeomorphism of $\R^2/\langle g \rangle$. Such an orientation reversing homeomorphism exists since $\R^2/\langle g \rangle$ is homeomorphic to $V_g\times \S^1$ there is an orientation reversing homeomorphism of $\S^1$. So, we can correctly assume that there is an orientation preserving homeomorphism $h: \R^2/\langle f \rangle\rightarrow \R^2/\langle g \rangle$. Finally, by the commutative diagram above, $U$ is orientation preserving. 
\end{proof}
Now we are ready to prove the main theorem of this paper.
 
\subsection{Proof of Theorem \ref{mainthm1}}
Lemma \ref{fiberbundle} shows that $\mathbb{R}^2/\langle f\rangle$ and $\mathbb{R}^2/\langle g\rangle$ are fiber bundles with $\mathbb{S}^1$ fibers over the leaf spaces $V_f$ and $V_g$ respectively. Since $f$ and $g$ both have finitely many fundamental regions, their leaf spaces $V_f$ and $V_g$ have finitely many edges. Therefore by Lemma \ref{contractable}, the leaf spaces $V_f$ and $V_g$ are contractible. Thus, $\mathbb{R}^2/\langle f\rangle$ and $\mathbb{R}^2/\langle g\rangle$ are homeomorphic to $V_f\times \mathbb{S}^1$ and $V_g\times \mathbb{S}^1$ respectively. Now suppose that the oriented plane foliations $F(f)$ and $F(g)$ are equivalent. Then by Haefliger-Reeb theory (see  \lemref{equivfoliations}) the corresponding leaf spaces $V_f$ and $V_g$ are homeomorphic, and hence $V_f\times \mathbb{S}^1$ and $V_g\times\mathbb{S}^1$ are also homeomorphic. Therefore we obtain that $\mathbb{R}^2/\langle f\rangle$ is homeomorphic to $\mathbb{R}^2/\langle g\rangle$ and hence by Proposition \ref{homeoconj} either $f$ is conjugate to $g$ or $f$ is conjugate to $g^{-1}$ in $\Homeo^{+}(\R^2)$.

To prove the converse, since $g$ and $g^{-1}$ have the same flowlines we may assume $g=hfh^{-1}$ where $h$ is an orientation preserving homeomorphism of $\R^2$. If $f$ is embedded in the flow $\{f^\rho\}$, then $hfh^{-1}$ is embedded in the flow $\{hf^\rho h^{-1}\}$. Therefore, $g$ is embedded in the flow $\{hf^\rho h^{-1}\}$ and by \lemref{darlinglemma} this means that $L_{h(x)}^g=\{hf^\rho h^{-1}(h(x)): \rho \in \R\}$. Since, $L_x^f=\{f^\rho(x): \rho \in \R\}$, we can obtain
$h(L_x^f)=h\{f^\rho(x): \rho \in \R\}=\{hf^\rho(x): \rho \in \R\}=\{hf^\rho h^{-1}(h(x)): \rho \in \R\}=L_{h(x)}^g$.
Therefore, $h(L_x^f)=L_{h(x)}^g$. Thus, by \defref{foliationequiv}, $F(f)$ and $F(g)$ are equivalent. 

\vspace{2mm}
\noindent
This completes the proof.

\qed

Our proof of \thmref{mainthm1} hinges on the contractibility of the leaf spaces $V_f$ and $V_g$ which has been obtained under the assumption that both $f$ and $g$ have finitely many fundamental regions. To the best of our knowledge the following is still open:

\begin{question}
Is the leaf space $V_f$ contractible for a free mapping $f$ with countably many fundamental regions?
\end{question}

An affirmative answer to the above question will lead to an immediate extension of our \thmref{mainthm1} to free mappings with countably many fundamental regions.


\bibliographystyle{alpha}

\begin{thebibliography}{GOS09}

\bibitem[And67]{andrea67}
Stephen~A. Andrea.
\newblock On homoeomorphisms of the plane which have no fixed points.
\newblock {\em Abh. Math. Sem. Univ. Hamburg}, 30:61--74, 1967.

\bibitem[Fra92]{jf92}
John Franks.
\newblock A new proof of the {B}rouwer plane translation theorem.
\newblock {\em Ergodic Theory Dynam. Systems}, 12(2):217--226, 1992.

\bibitem[FS55]{fs55}
N.~J. Fine and G.~E. Schweigert.
\newblock On the group of homeomorphisms of an arc.
\newblock {\em Ann. of Math. (2)}, 62:237--253, 1955.

\bibitem[GOS09]{goas09}
Nick Gill, Anthony~G. O'Farrell, and Ian Short.
\newblock Reversibility in the group of homeomorphisms of the circle.
\newblock {\em Bull. Lond. Math. Soc.}, 41(5):885--897, 2009.

\bibitem[HR57]{hr57}
Andr\'{e} Haefliger and Georges Reeb.
\newblock Vari\'{e}t\'{e}s (non s\'{e}par\'{e}es) \`a une dimension et
  structures feuillet\'{e}es du plan.
\newblock {\em Enseign. Math. (2)}, 3:107--125, 1957.

\bibitem[HR22]{s22}
Andr\'{e} Haefliger and Georges Reeb.
\newblock One dimensional non-{H}ausdorff manifolds and foliations of the plane.
\newblock{\em Séminaires et Congrès, Société mathématique de France}, Vol: Geometric methods in group theory: papers dedicated to Ruth Charney (original work published 1957 in French, translated by Gangotryi Sorcar), to appear.

\bibitem[Jon72]{jones72}
Gary~D. Jones.
\newblock The embedding of homeomorphisms of the plane in continuous flows.
\newblock {\em Pacific J. Math.}, 41:421--436, 1972.

\bibitem[Law74]{lawson74}
H.~Blaine Lawson, Jr.
\newblock Foliations.
\newblock {\em Bull. Amer. Math. Soc.}, 80:369--418, 1974.

\bibitem[Le06]{polish}
Zbigniew Le\'{s}niak.
\newblock On parallelizability of flows of free mappings.
\newblock {\em Aequationes Math.}, 71(3):280--287, 2006.

\bibitem[OS15]{fs15}
Anthony~G. O'Farrell and Ian Short.
\newblock {\em Reversibility in dynamics and group theory}, volume 416 of {\em
  London Mathematical Society Lecture Note Series}.
\newblock Cambridge University Press, Cambridge, 2015.

\end{thebibliography}

\end{document}